\documentclass[a4paper]{article}

\usepackage{amsmath}

\usepackage{amsthm}
\usepackage[dvips]{graphicx}

\usepackage{amscd}
\usepackage{epsfig}
 
\usepackage{color}

\usepackage{psfrag}
\usepackage{srcltx}

\usepackage{amssymb}
\usepackage{latexsym}
\newtheorem{Theorem}{Theorem}

\newtheorem{Definition}[Theorem]{Definition}
\newtheorem{Proposition}[Theorem]{Proposition}
\newtheorem{Lemma}[Theorem]{Lemma}
\newtheorem{Corollary}[Theorem]{Corollary}
\newtheorem{Remark}[Theorem]{Remark}

\newtheorem*{Assumption}{Assumption}

\begin{document}

\title{Conway spheres as ideal points of the character variety}

\author{Luisa Paoluzzi and Joan Porti\footnote{Partially supported by the 
Spanish Micinn/Feder through grant MTM2009-07594 and prize ICREA ACADEMIA 
2008}} 
\date{\today}

\maketitle

\begin{abstract}
\vskip 2mm

We prove that any Bonahon-Siebenmann family of Conway spheres for a hyperbolic
link is associated to an ideal point of the character variety of the link.

\noindent\emph{AMS classification: }  Primary 57M25; Secondary 20C99;  57M50.

\vskip 2mm

\noindent\emph{Keywords:} 
Conway spheres, hyperbolic links, Culler-Shalen theory, character varieties.

\end{abstract}

\section{Introduction}
\label{section:introduction}

Essential surfaces play a central role in the understanding of $3$-dimensional 
manifolds. In their seminal work on character varieties \cite{CullerShalen}, 
Culler and Shalen provided a powerful tool to find essential surfaces in 
$3$-manifolds. These are associated to ideal points of the character variety. 
However, not all essential surfaces can be detected this way: examples can 
be found in \cite{SchanuelZhang}.

In this paper, we consider a very special family of essential surfaces, that is
\emph{Conway spheres}. These are spheres embedded in ${\mathbf S}^3$ which meet 
a link transversally in four points and whose intersections with the link 
exterior are essential. Here by \emph{essential} we mean incompressible, 
$\partial$-incompressible and non boundary parallel.

In the following, given a Conway sphere $C$, we shall denote by $C'$ the
planar surface which is the intersection of $C$ with the exterior of the link.

\begin{Theorem}
\label{theorem:onesphere}
Assume $L$ is a hyperbolic link admitting a unique Conway sphere $C$ up to 
isotopy. Then $C'$ corresponds to an ideal point of the character variety of 
the exterior of $L$.
\end{Theorem}

For hyperbolic links admitting more than one Conway sphere, the above result
generalises to the following:

\begin{Theorem}
\label{theorem:morespheres}
Let $L$ be a hyperbolic link, and let $C_1,\dots,C_k$ be a Bonahon-Siebenmann
family of Conway spheres. Then the surface $C'_1\cup\cdots\cup C'_k$ 
corresponds to an ideal point of the character variety of the exterior of $L$.
\end{Theorem}

Recall that the \emph{Bonahon-Siebenmann family of toric sub-orbifolds} 
\cite{FrancisLarry} is the equivalent for orbifolds of the 
Jaco-Shalen-Johannson family for manifolds. Remark that, once the orders of 
ramification of the components of $L$ are assigned, the Bonahon-Siebenmann 
family is unique up to orbifold isotopy, as is the Jaco-Shalen-Johannson 
family, but the family can change for different orders of ramification. In our 
situation, since the link is hyperbolic, all elements of a Bonahon-Siebenmann 
family are Conway spheres which meet only components with order of ramification 
equal to $2$. 

Assume the hypotheses of Theorem \ref{theorem:onesphere} are fulfilled. Let
$M_1$ and $M_2$ be the components of ${\mathbf S}^3\setminus C$, and $M'_1$
respectively $M'_2$ their intersections with the exterior of the link. The key
idea in the proof of Theorem \ref{theorem:onesphere} is to consider the
projections of the character varieties for $M'_1$ and $M'_2$ into the character
variety for $C'$. Their intersection is a curve containing a distinguished
point. Such point satisfies the following properties: 

\begin{itemize}

\item It is a reducible character;

\item It is a limit point of irreducible characters in the intersection of the
two projections.

\end{itemize}

These conditions imply that the irreducible characters in the intersection
correspond to unique characters for the link exterior. However, these
characters do not have a limit in the character variety for the link exterior.
This means that our distinguished point is an ideal point for the link
exterior, even if it corresponds to a character of both $M_1'$ and $M_2'$.

Remark that there may exist representations of the link exterior which 
restricted to $M_i$ induce the same distinguished characters. For instance, 
this is the case if the link exterior is obtained as the double of $M_1'$, 
for in this case each representation for $M_1'$ can be extended to a 
representation of the link exterior, simply by doubling.

The proof of Theorem~\ref{theorem:morespheres} relies basically on the same 
idea seen for Theorem~\ref{theorem:onesphere}: some extra care is however 
needed to deal with several components at the same time.

The paper is organised as follows: the next two sections provide background
material on representation and character varieties 
(Section~\ref{section:background}) and on Culler-Shalen theory
(Section~\ref{section:cullershalen}). Section~\ref{section:C'} is devoted to 
the study of the character variety of the surface $C'$ and, more precisely, of 
a specific subvariety we shall be working with. In
Section~\ref{section:cutting} we shall discuss properties of the pieces 
obtained by cutting the link exterior along Conway spheres and their character 
varieties. In the last two sections the proofs of
Theorem~\ref{theorem:onesphere} (Section~\ref{section:thm1}) and 
Theorem~\ref{theorem:morespheres} (Section~\ref{section:bs families}) will be 
given.

\section{Background on varieties of representations}
\label{section:background}

The variety of representations of a connected compact manifold $N^n$ of 
dimension $n$ is the set of representations of its fundamental group in 
$SL(2,\mathbb{C})$:
$$
R(N^n)=\hom(\pi_1(N^n),SL(2,\mathbb{C})).
$$
Since we assume that $N^n$ is compact, $\pi_1(N^n)$ is finitely generated and 
thus $R(N^n)$ can be embedded in a product $SL(2,\mathbb{C})\times\cdots\times 
SL(2,\mathbb{C})$ by mapping each representation to the image of a generating 
set. In this way $R(N^n)$ is an affine algebraic set, whose defining 
polynomials are induced by the relations of a presentation of $\pi_1(N^n)$.
This structure is independent of the choice of presentation of $\pi_1(N^n)$, 
cf.~\cite{LubotzkyMagid}.

Given a representation $\rho\in R(N^n)$, its character is the map
$\chi_{\rho}:\pi_1(N^n)\to \mathbb C$ defined by 
$\chi_{\rho}(\gamma)=\operatorname{trace}(\rho(\gamma))$, 
$\forall\gamma\in\pi_1(N^n)$. The set of all characters is again an affine 
algebraic set, denoted by $X(N^n)$ and called the \emph{variety of characters 
of $N^n$}, even if it is often not irreducible \cite{AcunaMontesinos}. The 
projection
$$
\pi:R(N^n)\to X(N^n)
$$
is the quotient in the sense of invariant theory of the action of 
$PSL(2,\mathbb{C})=SL(2,\mathbb{C})/\mathcal{Z}(SL(2,\mathbb{C}))$ by 
conjugation: namely, if $ \mathbb{C}[R(N^n)]^{PSL(2,\mathbb{C})}$ denotes the 
$\mathbb{C}$-subalgebra of the polynomial functions on $R(N^n)$ which are 
$PSL(2,\mathbb{C})$-invariant, then
$$
\mathbb{C}[R(N^n)]^{PSL(2,\mathbb{C})}=\pi^*(\mathbb{C}[X(N^n)]),
$$
cf.~\cite{CullerShalen}. 

When $N^n$ is not connected, if $N^n=N^n_1\cup\cdots\cup N^n_k$ is the 
decomposition into connected components, one defines
$$
R(N^n)=R(N^n_1)\times \cdots \times R(N^n_k),
$$
and
$$
X(N^n)=X(N^n_1)\times \cdots \times X(N^n_k).
$$

\medskip

\begin{Definition}
 A representation $\rho\in R(N^n)$ is called \emph{reducible} if it has an 
invariant line in $\mathbb C^2$, namely if it is conjugate to a representation 
whose image consists of upper triangular matrices. Otherwise a representation 
is called \emph{irreducible}. 
\end{Definition}

\begin{Lemma}\cite{CullerShalen}
\label{Lemma:irreduciblerestriction}
Given $\rho,\rho'\in R(N^n)$, if $\chi_{\rho}= \chi_{\rho'}$ and $\rho$ is 
irreducible, then $\rho$ and $\rho'$ are conjugate; in particular $\rho'$ is
irreducible.

In addition, the projection $\pi:R(N^n)\to X(N^n)$ induces a holomorphic fibre 
bundle of the set of irreducible representations over the set or irreducible 
characters.
\end{Lemma}

According to Lemma~\ref{Lemma:irreduciblerestriction}, it makes sense to call a
character \emph{reducible} (respectively \emph{irreducible}) if it is
associated to a reducible (respectively irreducible) representation.
In addition, it is proved in \cite{CullerShalen} that the set of irreducible 
characters is Zariski open.

\medskip

Now assume that $N^3$ is a $3$-manifold and $S\subset N^3$ is an essential 
connected surface, that separates $N^3$ into two components $N^3_1$ and 
$N^3_2$. Then $\pi_1(N^3)$ is an amalgamated product
$$
\pi_1(N^3)=\pi_1(N^3_1)*_{\pi_1(S)}\pi_1(N^3_2).
$$
Let $r_i$ and $r_i'$ denote the restriction maps, induced by the natural
inclusions, in the following commutative diagram:
$$
\begin{CD}
X(N^3)   @>{r_1}>> X(N^3_1) \\
@V{r_2}VV  @VV{r_1'}V\\
X(N^3_2) @>{r_2'}>> X(S) 
\end{CD}
$$

\begin{Lemma}
\label{lemma:glueirreducible}
Let $\chi_1\in X(N^3_1)$ and $\chi_2\in X(N^3_2)$ be characters such that  
$r_1'(\chi_1)= r_2'(\chi_2)$ is irreducible. Then there exists a unique 
$\chi\in X(N^3)$ satisfying $r_1(\chi)=\chi_1$ and $r_2(\chi)=\chi_2$.
\end{Lemma}

\begin{proof}
Take representations $\rho_i\in R(N^3_i)$ such that $\chi_{\rho_i}=\chi_i$.  
Since $r_1'(\chi_1)= r_2'(\chi_2)$ is irreducible, after conjugation (see
Lemma~\ref{Lemma:irreduciblerestriction}), we may assume that the restrictions 
of $\rho_1$ and $\rho_2$ to $\pi_1(S)$ coincide. Thus they define a 
representation on the amalgamated product 
$\pi_1(N^3)=\pi_1(N^3_1)*_{\pi_1(S)}\pi_1(N^3_2)$. In addition, this 
representation is unique up to conjugacy, because the restriction of $\rho_i$ 
to $\pi_1(S)$ is irreducible and has no centraliser.
\end{proof}

Next assume that the essential surface $S\subset N^3$ is not connected. Let 
$S_1\ldots,S_k$ be its components, which are non-parallel, and assume that they 
split $N^3$ into $k+1$ components $N^3_1,\ldots N^3_{k+1}$, i.e. each $S_l$
is separating. The proof of Lemma~\ref{lemma:glueirreducible} can be extended 
by induction to obtain the following:

\begin{Lemma}
\label{lemma:glueseveralirreducible}
Let $\chi_1\in X(N^3_1), \ldots , \chi_{k+1}\in X(N^3_{k+1})$ be characters 
whose restrictions to 
%the components of 
$N^3_i\cap N^3_j$ coincide and are 
irreducible, whenever this intersection is nonempty. Then there exists a unique 
$\chi\in X(N^3)$ whose restriction to $N_i^3$ is $\chi_i$, for 
$i=1,\ldots,k+1$.
\end{Lemma}

\section{Culler-Shalen theory}
\label{section:cullershalen}

Culler-Shalen theory associates essential surfaces to ideal points in
the projective completion of 
a curve of characters. Recall that a surface is called essential if it is 
incompressible, $\partial$-incompressible and not boundary parallel. The 
surface does not need to be connected, and from now on we assume that its 
components are pairwise nonparallel.

Given an algebraic curve $\mathcal C$ in $X(N^3)$, let $\mathcal C_0$ be a 
smooth model of the projective completion of $\mathcal C$. An ideal point $x$ 
is a point in $\mathcal C_0\setminus\mathcal C$, namely a point at infinity. 
The essential surface associated to $x$ is constructed in essentially two main 
steps:
\begin{itemize}
\item 
\emph{From ideal points to representations over a field with a discrete 
valuation}. 
The ideal point $x\in \mathcal C_0\setminus\mathcal C$ defines a 
discrete valuation $v_x$ on the function field $\mathbb{C}(\mathcal C)$ as 
follows. For any rational function $f\in \mathbb{C}(\mathcal C)$,
$v_x(f)=n\geq 0$ if $x$ is a zero of $f$ of order $n$, and $v_x(f)=-n < 0$ if 
$x$ is a pole of order $n$. Given a curve $\tilde {\mathcal C}\subset R(N^3)$ 
that projects to $ \mathcal C$, the field $F=\mathbb{C}(\tilde {\mathcal C})$ 
is a finite extension of $\mathbb{C}(\mathcal C)$, and (a multiple of) the 
valuation $v_x$ extends to a discrete valuation $\tilde v\!:\! F\to \mathbb Z$. 
Construct the tautological representation $P:\pi_1(N^3)\to SL(2, F)$, by 
viewing the entries of a matrix in $SL(2,\mathbb C)$ as polynomial 
functions on $\tilde{\mathcal C}$. Notice that there are elements 
$\gamma\in\pi_1(N^3)$ that satisfy $ \tilde v( P(\gamma))<0$.

\item \emph{The action on the Bass-Serre tree}.  
In \cite{Serre}, J.-P. Serre constructed the Bass-Serre tree $T$ associated to 
$SL(2,F)$, where $F$ is a field with a discrete valuation 
$\tilde v: F\to \mathbb Z$. The group $SL(2,F)$ acts on this tree without 
reversing the orientation of its edges. This action has the property 
that the stabilisers of vertices are precisely the subgroups that can be 
conjugated into $SL(2,R_{\tilde v})$, where 
$R_{\tilde v}=\{f\in F\mid \tilde v(f)\geq 0\}$ is the ring of the valuation. 
Since there are elements $\gamma\in\pi_1(N^3)$ satisfying 
$\tilde v (P(\gamma))<0$, the induced action of $\pi_1(N^3)$ is not contained 
in a vertex stabiliser. Then, one considers an equivariant map from the 
universal covering of $N^3$ to the tree:
%, $f:\tilde N^3\to T$, 
to obtain an essential surface, it suffices to take the 
inverse image of midpoints of edges of $T$ and to render it essential in an 
equivariant way. The reader is referred to \cite{CullerShalen} for details.
\end{itemize}

Notice that the essential surface constructed above is \emph{not unique}: 
indeed, even the choice of the equivariant map from the universal covering 
of $N^3$ to $T$ is not unique.

Now we want to establish a sufficient condition for a surface to be associated 
to an ideal point.

Let $S\subset N^3$ be an essential surface with components $S_1,\ldots,S_k$, 
each one separating, so that $N^3_1,\ldots,N^3_{k+1}$ are the complementary 
components of $S$. Let $\mathcal C\subset X(N^3)$ be an algebraic curve and 
$\{\chi_n\}$ be a sequence of characters in $\mathcal C$. Assume that the 
sequence of the restrictions to $N^3_i$, $\{\chi^i_n\}$, converges to the 
irreducible character $\chi^i_{\infty}$. Let $\rho_i\in R(N^3_i)$ be the 
representation whose character is $\chi_{\rho_i}=\chi^i_{\infty}$, for 
$i=1,\ldots, k+1$. 

\begin{Lemma}
\label{lemma:suffideal}
Assume that whenever $S_l= N_i^3\cap N_j^3$, the restriction of $\rho_i$ to 
$S_l$ is \emph{not conjugate} to the restriction of $\rho_j$ to $S_l$.
Then a subsequence of $\{\chi_n\}$ converges to an ideal point $\chi_{\infty}$ 
of $\mathcal C$, and $S$ is a surface associated to this ideal point.
\end{Lemma}
 
\begin{proof}
Assume first that $k=1$, i.e. $S=S_1$ is connected. Up to a subsequence, 
$\{\chi_n\}$ converges to either a point in $\mathcal C$ or to an ideal point.
We want to see that it converges to an ideal point: seeking a contradiction, 
assume that it converges to $\chi_{\infty}\in \mathcal C\subset X(N^3)$. In 
particular $\chi_{\infty}=\chi_{\rho_{\infty}}$ for some 
$\rho_{\infty}\in R(N^3)$ and since $\chi_{\infty}$ restricted to $N^3_i$ is the
irreducible character $\chi_\infty^i$, then $\rho_{\infty}$ restricted to $N^3_i$ 
is conjugate to $\rho_i$ (see Lemma~\ref{Lemma:irreduciblerestriction}). In 
particular the restrictions of $\rho_1$ and $\rho_2$ to $S_1$ are conjugate, 
leading to a contradiction. Thus $\chi_n\to x$, an ideal point. Now consider 
the Bass-Serre tree $T$ associated to this ideal point. By construction, since 
the valuation restricted to $N^3_i$ is non-negative, $P(\pi_1(N^3_i))$ is contained 
in the stabiliser of a vertex $v_i$ of $T$ (where $P$ is the tautological 
representation described above). Notice that $v_1\neq v_2$, otherwise the point 
$x$ would not be ideal. 

Consider the graph of groups $G$, with two vertices, with vertex groups 
$\pi_1(N^3_1)$ and $\pi_1(N^3_2)$ respectively, and with one edge between them, 
with edge group $\pi_1(S)$. This graph is dual to the decomposition along $S$. 
Now construct an equivariant map from $\tilde G$ to $T$, by mapping the $i$-th 
vertex to $v_i$, the edge of $G$ to the unique path in $T$ joining them (note 
that this path can consists of more than one edge).
 Compose it with 
the obvious equivariant map from $\tilde N^3$ to $\tilde G$:
$$
\tilde N^3\to\tilde G\to T.
$$
Taking into account this very action, and given the fact that $S$ is essential, 
it is clear that the surface obtained by taking the inverse image of midpoints 
of edges and eliminating parallel copies is precisely $S$.

The general case follows easily, since each component $S_l$ separates.
\end{proof}

\begin{Remark}
\label{rem:nonconj}
In the hypothesis of Lemma~\ref{lemma:suffideal} the representations $\rho_i$ 
and $\rho_j$ restricted to $S_l$ are not conjugate, but they have the same 
character: $\chi_i\vert_{S_l}=\chi_j\vert_{S_l}$. This follows from the fact 
that this is the case for the restrictions of $\chi_{n}$ to $N_i^3$ and $N_j^3$ and that the corresponding sequences converge.

This situation can occur when the restrictions of $\rho_i$ and $\rho_j$ to 
$S_l$ are reducible, and this is precisely what happens in our applications.
\end{Remark}

\section{The character variety for $C'$}
\label{section:C'}

Recall that, for a given Conway sphere $C$, $C'$ denotes the four-holed sphere
which is the intersection of $C$ with the link exterior. The fundamental group
of $C'$ is a free group of rank $3$. For our purposes, we choose the
presentation
$$\pi_1(C')=\langle \mu_1,\mu_2,\mu_3,\mu_4\mid \mu_1\mu_2\mu_3\mu_4 \rangle$$
where the $\mu_i$'s correspond to the four peripheral elements.

According to \cite{AcunaMontesinos} the $SL(2,{\mathbb C})$-character variety
for a free group of rank $3$ is a hypersurface in ${\mathbb C}^7$, whose
coordinates correspond to the traces of the images of $\mu_i$, $i=1,2,3$,
$\mu_i\mu_j$, $1\le i<j\le3$ and $\mu_1\mu_2\mu_3=\mu_4^{-1}$. 

We will only be interested in characters coming from representations induced
by those of $M_1'$ and $M_2'$ and that potentially extend to the link 
exterior. In particular, when the link is a knot, we have that the traces of
the $\mu_i$'s must be the same. We shall then only need to consider the 
intersection of the character variety with a $4$-plane. Even in the case where
the link has more than one component, we shall only consider this subvariety,
which is again a hypersurface but in ${\mathbb C}^4$. 

The hypersurface obtained by imposing all traces of meridians to be equal is
\begin{equation}
\label{eq:hypersurface}
\mathcal Y=\{(x,y,z,t)\in\mathbb C^4\mid (t^2-(x+y+z-2))^2=(2-x)(2-y)(2-z)\},
\end{equation}
where $t$ represents the trace of the image of $\mu_i$, $i=1,2,3$, and
$\mu_1\mu_2\mu_3=\mu_4^{-1}$, while $x$, $y$ and $z$ those of $\mu_1\mu_2$,
$\mu_1\mu_3$ and $\mu_2\mu_3$ respectively.

One can prove that $\mathcal Y$ is an irreducible hypersurface whose singular set
consists of the point $(-2,-2,-2,0)$ together with three
one-dimensional components. The three singular curves meet at the
character corresponding to the trivial representation and contain the
points $(-2,2,2,0)$,  $(2,-2,2,0)$ and $(2,2,-2,0)$ respectively.

Note that, since $\pi_1(C')$ is a free group, the projection from 
$X(C', SL(2,\mathbb{C}))$ to $X(C', PSL(2,\mathbb{C}))$ is a surjection. 
The points of $\mathcal Y$ we are interested in correspond to characters of 
lifts of parabolic representations in $PSL(2,\mathbb{C})$, where the meridians 
are rotations of angle $\pi$. In particular their holonomies are conjugate to 
the matrix
$$
\pm\begin{pmatrix}
      i & 0 \\
  0 & -i
   \end{pmatrix}.
$$
According to the possible lifts (i.e. choices of signs), the possible values 
for $(x,y,z,t)$ are
$
(-2,-2,-2,0)
$,
$
(-2,2,2,0)
$,
$
(2,-2,2,0)
$ or
$
(2,2,-2,0)
$.
In fact, we shall show in Lemma~\ref{Lemma:liftsmatch} that the point 
$(x,y,z,t)=(-2,-2,-2,0)$ does not occur. We shall need to work in 
neighbourhoods of these points, thus considering $\mathbb C$-analytic 
varieties. We follow the usual notation and call germ at a point the analytic 
variety defined around the point, without specifying the neighbourhood. 

\begin{Lemma}
\label{Lemma:locallyirreducible} 
The analytic germ of the hypersurface $\mathcal Y$ in ${\mathbb C}^4$ of 
equation (\ref{eq:hypersurface}) at the points $(x,y,z,t)=(-2,2,2,0)$, 
$(2,-2,2,0)$ and $(2,2,-2,0)$ is irreducible.
\end{Lemma}

\begin{proof}

We shall analyse the germ at the point $(-2,2,2,0)$: for symmetry reasons, it 
is sufficient to consider this case.
We apply the following algebraic change of variables: 
$\xi=t^2-x-2$,
$\upsilon=2-y$,
$\zeta=2-z$ and
$\tau=t$
so that the coordinates of our point become $(0,0,0,0)$ and the equation
\begin{equation}
\mathcal Y=\{(\xi,\upsilon,\zeta,\tau)\in\mathbb C^4\mid 
(\xi-\upsilon-\zeta)^2=(\xi-\tau^2+4)\upsilon\zeta \}.
\end{equation}

We now consider the algebraic hypersurface in $\mathbb C^4$ defined by the 
equation
\begin{equation}
\mathcal H=\{(w,u,v,t)\in\mathbb C^4\mid u^2+v^2+uvw-w^2-t^2+4=0\}.
\end{equation}
We can define a regular morphism $\mathcal H\longrightarrow \mathcal Y$ as
follows: 
$$(w,u,v,t)\mapsto(\xi=u^2+v^2+uvw,\upsilon=u^2,\zeta=v^2,\tau=t).$$
Note that we have $\xi-\upsilon-\zeta=uvw$ and 
$\xi-\tau^2+4=u^2+v^2+uvw-t^2+4=w^2$. It is easy to check that this morphism is
finite to one, hence proper, and surjective. The preimage in $\mathcal H$ of
$(0,0,0,0)$ consists of two points: $(\pm 2,0,0,0)$. It is trivial to see that 
these are smooth points of $\mathcal H$, for the derivative with respect to $w$
is non zero. It is now clear that the analytic germ of $\mathcal Y$ at
$(-2,2,2,0)$ has at most two components. Irreducibility follows from the fact 
that the involution $(w,u,v,t)\mapsto(-w,u,-v,t)$ acting on $\mathcal H$ 
exchanges the two points $(\pm 2,0,0,0)$ without changing the regular morphism.

\end{proof}

\section{Cutting off along Conway spheres}
\label{section:cutting}

Cut off the exterior of the link along a family of pairwise non-parallel Conway 
spheres $C_1,\ldots,C_k$, obtaining components $M_1,\ldots,M_{k+1}$. Let 
$\mathcal O_i$ denote the orbifold with underlying space $M_i$, branching locus 
$L\cap M_i$. We require that the ramification indices are equal to $2$ along 
the components that meet some Conway sphere.

\begin{Assumption}
The orbifolds $\mathcal O_1, \ldots,\mathcal O_{k+1}$ are geometric, either 
hyperbolic with finite volume or Seifert fibred with hyperbolic base.
\end{Assumption}

We shall show in Section~\ref{section:thm1} that under the hypotheses of 
Theorem~\ref{theorem:onesphere} this assumption is satisfied for an appropriate 
choice of the ramification. Note that, under the hypotheses of 
Theorem~\ref{theorem:morespheres}, all orbifolds $\mathcal O_i$ (where the 
orders of ramification are those that determine the Bonahon-Siebenmann family) 
are geometric by definition, but the Seifert fibred ones do not have 
necessarily a hyperbolic base. We shall assume that this does not happen, and 
we shall see later how to avoid this particular situation. Note that if the 
link is a knot, all bases of Seifert fibred components must be hyperbolic.

Let $C_l$ be a Conway sphere and $M_i$ a component adjacent to it. Let us 
denote, as usual, $C_l'$ and $M'_i$ their respective intersections with the 
exterior of the link. The holonomy of the hyperbolic structure of 
$\mathcal O_i$, or of its hyperbolic base, if $\mathcal O_i$ is Seifert fibred, 
restricts to a $PSL(2,\mathbb{C})$-representation of $C'_l$ that, up to 
conjugacy, maps each $\mu_i$ to an element of the form
\begin{equation}\label{meridian}
\pm\begin{pmatrix}
    i & * \\
    0 & - i
   \end{pmatrix}.
\end{equation}
The character of its lift to $SL(2,\mathbb C)$  will be denoted by $\chi_0$ 
and has coordinates 
$(x,y,z,t)=
(-2,-2,-2,0)
$,
$
(-2,2,2,0)
$,
$
(2,-2,2,0)
$ or
$
(2,2,-2,0)
$.
In fact the case $(-2,-2,-2,0)$ will be ruled out in 
Lemma~\ref{Lemma:liftsmatch}. Notice that $\chi_0$ is \emph{reducible}, in 
particular there might be several non conjugate representations of $C'_l$ whose 
characters are $\chi_0$.

Let $\chi_i\in X(M_i')$ be the character of a lift of the representation 
induced by the holonomy of the complete hyperbolic structure on $\mathcal O_i$, 
or the hyperbolic structure of its base when it is Seifert fibred. The 
existence of the lift is due to \cite{culler}. Recall that, even if $\chi_0$ is 
reducible, $\chi_i$ is \emph{irreducible}.

When $C_l$ is a component of $\partial M_i$, let 
$$
r_{il}: X(M_i')\to X(C_l')
$$ 
denote the restriction map.

\begin{Lemma}
\label{Lemma:liftsmatch} 
The lifts $\chi_i\in X(M_i')$, for $i=1,\ldots,k+1$ can be chosen so that
$$
r_{il}(\chi_i)=r_{jl}(\chi_j),
$$
whenever $M_i\cap M_j= C_l$. 

In addition, the coordinates of $r_{il}(\chi_i)=r_{jl}(\chi_j)$ are 
$(x,y,z,t)=
(-2,2,2,0)
$,
$
(2,-2,2,0)
$ or
$
(2,2,-2,0)
$ (i.e. the case $(-2,-2,-2,0)$ does not occur).
\end{Lemma}

\begin{Remark}
\label{rem:warningnotconj}
We will show in Remark~\ref{rem:no match} that the restricted representations 
are not conjugate, even if they have the same character, 
cf.\ Remark~\ref{rem:nonconj}. This will be used to apply 
Lemma~\ref{lemma:suffideal}. 
\end{Remark}

\begin{proof}[Proof of Lemma~\ref{Lemma:liftsmatch}]
Two different lifts may differ by a change of sign at the meridians of the arcs 
$L\cap M_i$. To make a consistent choice, we fix an \emph{orientation} of the 
components of $L$. 

For each meridian $\mu\in\pi_1(M_i')$, choose an isometry 
$a\in PSL(2,\mathbb C)$ that maps 
the oriented line from $0$ to $\infty$
to the end-points of the oriented axis of 
$\rho_i(\mu)$, in the upper half 
space model for $\mathbf H^3$. In particular, a lift $\rho_i$ of the holonomy 
must satisfy:
\begin{equation}
a^{-1}\rho_i(\mu) a=\epsilon 
  \begin{pmatrix}
  i & 0 \\   0 & -i 
  \end{pmatrix}
\qquad\textrm{ for some } \epsilon= \pm 1.
 \label{eqn:orientations}
\end{equation}
We fix the following convention: we choose $\epsilon=1$ if the orientations of 
$L$ and $\mu$ induce the standard orientation of $\mathbf S^3$ and 
$\epsilon=-1$ otherwise, cf.\ Figure~\ref{fig:orientations}. A different choice 
of isometry $a\in PSL(2,\mathbb C)$ differs by an isometry that preserves the 
oriented line from $0$ to $\infty$, hence it commutes with the isometry in 
(\ref{eqn:orientations}). Note also that the inverse in $SL(2,\mathbb C)$ of 
the matrix in the above identity coincides with its opposite, which is 
in accordance with the chosen convention.

\begin{figure}
\begin{center}
{
\psfrag{m}{$\mu$}
\psfrag{0}{}
\psfrag{i}{$L                                           $}
\psfrag{ep}{$\epsilon=+1$}
\psfrag{em}{$\epsilon=-1$}
\includegraphics[height=4cm]{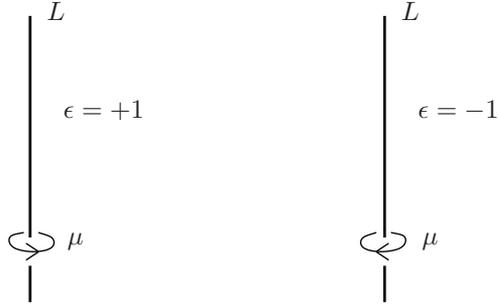}
}
\end{center}
\caption{The choice of $\epsilon=\pm 1$ according to the orientation.}
\label{fig:orientations}
\end{figure}

We need to check that this choice is consistent. We identify $M_i'$ with the 
exterior of an embedded graph in $\textbf S^3$ so that its vertices are 
$4$-valent, and they correspond precisely to the Conway spheres in 
$\partial M_i$. In this way, one can compute the fundamental group with the 
Wirtinger method, and we obtain that $\pi_1(M_i')$ is generated by meridians of 
the arcs, and there are two kinds of relations:
\begin{itemize}
\item[$r_{crossings}$] The usual conjugacy relations corresponding to 
transverse crossings of a projection, as for links.
\item[$r_{vertices}$] Relations corresponding to the vertices, i.e. to Conway 
spheres. Using the presentation at the beginning of Section~\ref{section:C'},
for $\mu_1,\mu_2,\mu_3,\mu_4\in\pi_1(C'_l)$, the relation is 
$\mu_1\mu_2\mu_3\mu_4=1$. 

\end{itemize}
As for links, this family of relations is not minimal, but they generate all 
relations. Let us check compatibility. First we deal with $r_{crossings}$, 
which are relations of conjugations between meridians. 
If $\mu,\mu'\in\pi_1(M_i')$ are two meridians for the same component of 
$L\cap M_i$ with the same orientation, then $\mu=\gamma^{-1}\mu'\gamma$ for 
some $\gamma\in \pi_1(M_i')$. In particular, for $a\in PSL(2,\mathbb C)$ as in 
Equation~\ref{eqn:orientations}, 
$$
a^{-1}\rho_i(\mu) a =(\rho(\gamma) a)^{-1}\rho_i(\mu') \rho(\gamma) a.
$$ 
This shows compatibility for relations $r_{crossings}$.

For $r_{vertices}$, look at the restriction of $\rho_i$ to $C'_l$. Up to 
conjugation, we may assume that the point in $\partial_{\infty}(\mathbf H^3)$ 
fixed by $\pi_1(C'_l)$ is $\infty$, hence 
\begin{equation}
 \label{eq:restriction}
\rho(\mu_i)= \epsilon_i \begin{pmatrix}
  i & * \\   0 & -i
  \end{pmatrix}, \qquad \epsilon_i=\pm 1.
\end{equation}
Since the intersection number between $L$ and  $C_l$ is zero, the set 
$\{\epsilon_1,\epsilon_2,\epsilon_3,\epsilon_4\}$ contains precisely 
twice $+1$ and twice $-1$. Consequently
$$
\rho(\mu_1)\rho(\mu_2)\rho(\mu_3)\rho(\mu_4)=\begin{pmatrix}
  i^4 & 0 \\   0 & (-i)^4
  \end{pmatrix}=\begin{pmatrix}
  1 & 0 \\   0 & 1
  \end{pmatrix}.
$$

Thus, the lift is coherent with all the relations of the Wirtinger 
representation (i.e.\ the lifted representation maps the relators to the 
identity matrix, instead of minus the identity).

We look again  at the restriction of $\rho_i$ to $C'_l$. Since $x$ is the trace 
of $\mu_1\mu_2$, from our convention and formula (\ref{eq:restriction}), it 
follows easily that $x=-2$ if the arcs of $\mu_1$ and $\mu_2$ cross $C_l$ in 
the same direction, and $x=2$ if they do it in opposite directions. Of course 
the same holds true for $y$ and $z$ with the respective oriented arcs. Since 
the values of $x$, $y$, $z$ and $t=0$ determine the character, this proves the 
compatibility.

To prove that the case $(-2,-2,-2,0)$ does not occur, notice that it requires 
that $L$ crosses at least three times $C_l$ in the same direction; this is 
impossible, since the intersection number between $L$ and  $C_l$ is zero.
\end{proof}

As in previous section, $\mathcal Y$ denotes the subset of $X(C')$ obtained by 
requiring that the traces of all meridians are the same. We use the notation 
$\mathcal Y_l\subset X(C'_l)$ to refer to the $l$-th component $C_l$. Recall 
that $\chi_0\in \mathcal Y_l$ denotes the character of $X(C_l')$ such that 
$r_{il}(\chi_i)=\chi_0$.

\begin{Lemma}\label{lemma:irrdeformations}
Let $\{\chi_s\in \mathcal Y_l\}_{s\in[0,\varepsilon)}$ be a continuous 
deformation of $\chi_0$ for which $t(s)\neq 0$ for all $s\in (0,\varepsilon)$. 
Assume that $\chi_s=r_{il}(\chi_{i,s})$ for all $s\in [0,\varepsilon)$, where 
$\{\chi_{i,s}\in X(M_i')\}_{s\in[0,\varepsilon)}$ is a continuous deformation 
of $\chi_i$. Then there exists an $s\in (0,\varepsilon)$ such that $\chi_s$ is 
irreducible.
\end{Lemma}

\begin{proof}

Assume by contradiction that all $\chi_s$ are reducible. Since $\chi_i$ is 
irreducible, and the set of irreducible characters is Zariski open, 
$\chi_{i,s}$ is irreducible $\forall s\in [0,\varepsilon)$. Thus, by 
Lemma~\ref{Lemma:irreduciblerestriction}, the deformation 
$\{\chi_{i,s}\in X(M_i')\}_{s\in[0,\varepsilon)}$ lifts to a continuous 
deformation of representations in $R(M_i')$. By considering its restriction to 
$\pi_1(C_l')$, we get, for each $s$, a (reducible) representation 
$\rho_{0,s}\in R(C_l')$ whose character is $\chi_s$, and which depends 
continuously on $s$.

Up to symmetry, we may assume that the coordinates of $\chi_0$ are 
$(x,y,z,t)=(-2,2,2,0)$. We have
$$
\rho_{0,s}(\mu_i)=
\begin{pmatrix}
    \lambda_i(s) & * \\
    0 & \lambda_i(s)^{-1}
   \end{pmatrix}
$$
and 
$\lambda_1(0)=\lambda_2(0)=\frac1{\lambda_3(0)}=\frac1{\lambda_4(0)}=\pm i$.
Recall that we are working in ${\mathcal Y}$, where the traces of the 
$\rho_{0,s}(\mu_i)$ are all the same. As a consequence,
$\lambda_i(s)+\lambda_i(s)^{-1}=\lambda_j(s)+\lambda_j(s)^{-1}$ for all $i$ and
$j$. This means that $\lambda_i(s)=\lambda_j(s)^{\pm 1}$ for all $i$, $j$ and 
$s\in (0,\varepsilon)$. For small $\varepsilon>0$, by continuity we must have 
$$
\lambda_1(s)=\lambda_2(s)=\frac1{\lambda_3(s)}=\frac1{\lambda_4(s)},
$$
for $s$ in a neighbourhood of $0$. By the relation between the 
$\lambda_i(s)$'s, we get $y^2(\chi_s)=z^2(\chi_s)=4$ and $x^2(\chi_s)$ is 
non-constant, as $t$ is non-constant either.

 We project the characters $\chi_{i,s}\in X(M_i')$ to $PSL(2,\mathbb C)$ and
we lift them again to $SL(2,\mathbb C)$, continuously on $s$. 
All 
possible different lifts can be realised by maps from $\pi_1(M'_i)$ to 
$\mathbb Z/2$, by composing one given representation with an abelian 
representation into 
$$
\left\{\pm \begin{pmatrix} 1 & 0 \\ 0 & 1 \end{pmatrix}\right\}.
$$
One can easily prove the existence of maps $\pi_1(M'_i)\to \mathbb Z/2$
that are non trivial on two meridians of $C_l'$ and are trivial on 
the other two meridians of $C_l'$. This will change the lift to one with 
coordinates $(2,-2,2,0)$, $(2,2,-2,0)$ or $(-2,-2,-2,0)$ (notice that here we 
only extend the character to $M_i'$ not to the whole $\mathbf S^3\setminus L$, 
and therefore the case$(-2,-2,-2,0)$ can occur). Assume first that the 
coordinates are $(x,y,z,t)=(2,-2,2,0)$. The previous discussion implies 
that $x^2(\chi_s)=z^2(\chi_s)=4$ and $y^2(\chi_s)$ is non-constant, hence a 
contradiction, because $x^2$, $y^2$ and $z^2$ are functions that do not depend 
on the lift from $PSL(2, \mathbb{C})$ to $SL(2, \mathbb{C})$.

The case $(x,y,z,t)=(2,2,-2,0)$ being analogous, we next assume that 
$(x,y,z,t)=(-2,-2,-2,0)$. We repeat the construction of the representation of 
$\rho_{0,s}$ as above, with the difference that now 
$$
\lambda_1(s)=\lambda_2(s)=\lambda_3(s)=\lambda_4(s),
$$ 
which implies that $x^2(\chi_s)=y^2(\chi_s)=z^2(\chi_s)$, leading again to a
contradiction.
\end{proof}

\section{A unique Conway sphere}
\label{section:thm1}

The aim of this section is to study the case when the link admits a unique
Conway sphere $C=C_1$. With the same notation as in the previous section, we 
shall denote by ${\mathcal O}_i$, $i=1,2$, the two orbifolds of the
decomposition. Since the Bonahon-Siebenmann family must be contained in
$\{C\}$, we see that the two orbifolds ${\mathcal O}_i$ are geometric. Note, 
moreover, that if both orbifolds are Seifert, their fibrations cannot
match, for else the base of the global fibration would be large
 and the link would then 
admit some other Conway sphere. In particular $C$ is precisely the only element
of the Bonahon-Siebenmann family.

\begin{Lemma}
\label{lemma:euclidean}
Assume ${\mathcal O}_i$ is a Seifert fibred orbifold, with Euclidean base. Then
the associated tangle consists of two straight vertical arcs plus an unknotted 
circle isotopic to the equator.
\end{Lemma}

\begin{proof}
The base of the fibration is a $2$-dimensional Euclidean orbifold. These are 
completely classified (see \cite{ThurstonBook}), and we can check that there 
are only two planar orbifolds with a unique boundary component (non planar 
bases cannot correspond to fibrations of the ball). The first one is obtained 
in the following way: quotient a ``pillowcase" by the reflection in a
great circle containing its four cone points and cut in half the resulting
orbifold. This way one obtains a rectangle with three mirrored sides and one 
boundary component, and two corner points (with angle $\pi/2$ and of dihedral 
type). The second one is obtained by quotienting a M\"obius band by a
reflection in a line orthogonal to its core. The result is a bigon with one 
mirrored side and one cone point of order $2$ in its interior (note that the
reflection fixes an extra point on the central curve of the M\"obius band). 
Both base orbifolds can also be seen as the quotient by two different 
reflections of a disc with two cone points of order $2$: the axis of the first 
reflection contains both cone points, while the second reflection exchanges
them. It is not hard to see that these bases correspond to the two different 
fibrations of the tangle described in the statement of the lemma, which are 
pictured in Figure \ref{fig:eucfib}. Note that the extra closed component of 
the tangle is a fibre of the second fibration, corresponding to the cone point 
of the base orbifold.
\end{proof}

\begin{figure}
\begin{center}
{
\includegraphics[height=4cm]{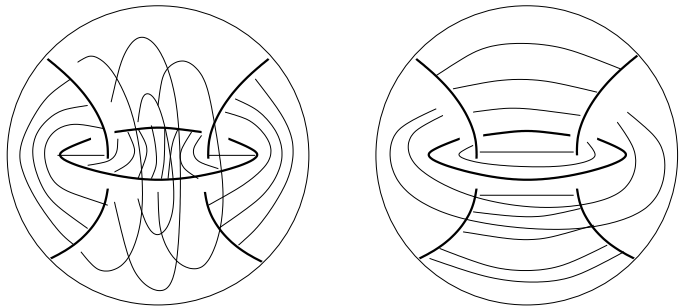}
}
\end{center}
   \caption{Two fibrations of the tangle with Euclidean base.}
\label{fig:eucfib}
\end{figure}

By putting a branching index $3$ on the closed component of the tangle of 
Lemma~\ref{lemma:euclidean} we get the following:

\begin{Corollary}
We can modify the branching indices of the orbifolds $\mathcal O_1$
and $\mathcal O_2$ so that:
\begin{itemize}

\item The singular components that meet the Conway sphere have branching index 
$2$;

\item The two resulting orbifolds, which will again be denoted by $\mathcal
O_1$ and $\mathcal O_2$, are either hyperbolic or Seifert fibred with 
hyperbolic base.
\end{itemize}
\end{Corollary}

\begin{Remark}\label{remark:ramification}
From now on $\mathcal O_i$, $i=1,2$, will denote the orbifold with underlying 
topological space $M_i$ and singular set $M_i\cap L$ with branching order $2$ 
on the components meeting the boundary and $3$ on the closed components. Remark 
that now $\mathcal O_i$, $i=1,2$, satisfies the conclusions of the above 
corollary. We shall denote by $\mathcal O_i'$, $i=1,2$, the compact orbifold 
obtained from $\mathcal O_i$ by drilling out open regular neighbourhoods of the 
two singular components of branching index $2$, i.e. those that meet the 
boundary.
\end{Remark}

Let $\chi_i\in X(\mathcal O'_i)$ be a lift of the character induced by the 
holonomy of the geometric structure of $\mathcal O_i$, chosen so that it
satisfies the compatibility condition established in 
Lemma~\ref{Lemma:liftsmatch}. Notice that Culler's theorem \cite{culler} allows 
to lift the representation of $\pi_1(M_i')$, and since $M_i'$ is obtained from 
$\mathcal{O}'_i$ by removing all components of order $3$, the lift extends from 
$M_i'$ to $\mathcal O'_i$, by changing the signs of the corresponding meridians
 as in the proof of Lemma~\ref{Lemma:liftsmatch}, 
if needed.

\begin{Lemma}\label{Lemma:dim3}
The character $\chi_i$ is a smooth point in $X(\mathcal O'_i)$. The traces of 
the two meridians and of a peripheral element of $\mathcal O_i$ define local 
coordinates in a neighbourhood of $\chi_i$.
\end{Lemma}

\begin{proof}

For a manifold admitting an irreducible representation, Theorem 5.6 in
Thurston's notes \cite{ThurstonNotes} states that the complex dimension of the
character variety is at least $-3e/2 +\theta$ where $e$ is the Euler characteristic
of the boundary of the manifold and $\theta$ the number of tori in the boundary.
If we fill some toric components with singular solid tori, in such a way that
the resulting orbifold still admits an irreducible representation, then
the same dimension bound still holds. This means that we can apply Thurston's
result to conclude that the variety of characters of $\mathcal O'_i$ has
dimension at least three. This gives a lower bound on the dimension.

To obtain an upper bound, we must work in the Zariski tangent space, i.e. the 
first cohomology group with coefficients in the Lie algebra twisted by the 
adjoint representation, that we denote by 
$H^1(\mathcal O'_i;\mathfrak{sl}(2,\mathbb C))$.

We first show that the variety of $PSL(2,\mathbb C)$-characters 
$X(\mathcal O_i,PSL(2,\mathbb C))$ is one-dimensional, locally parametrised by 
deformations of its boundary $\partial\mathcal O_i$ (since 
$ \pi_1(\mathcal O_i)$ has 2-torsion we must work in $PSL(2,\mathbb C)$ instead 
of $SL_2(\mathbb C)$). Moreover the Zariski tangent space at the holonomy 
character is also one dimensional, namely for the holonomy representation we 
have: $H^1(\mathcal O_i,\mathfrak{sl}(2,\mathbb C))\cong \mathbb C$.
In the hyperbolic case, this follows from the proof of Thurston's hyperbolic
Dehn filling theorem \cite{ThurstonNotes, KapovichBook} (see also
\cite{BoileauPorti} for the precise statement for orbifolds.) In the Seifert
fibred case with hyperbolic base, since all irreducible representations factor
through the base, it suffices to determine the variety of characters of the
base (cf.~\cite[Lemma~3.4]{Porti}). This is well-known, using the fact that the 
base is a small orbifold, cf.\ \cite{Goldman}.

After lifting, the $PSL(2,\mathbb C)$-character variety of $\mathcal O_i$ can 
be identified to the the subvariety of $X(\mathcal O_i')$ obtained by imposing 
that the traces of the meridians are zero, that is by intersecting the latter 
variety with two hyperplanes. Since $X(\mathcal O_i)$ is (non empty) of
dimension $1$, standard results on the dimension of intersections of algebraic
varieties imply the dimension of $X(\mathcal O_i')$ is at most
three. Similarly, since the tangent space to $X(\mathcal O_i,PSL(2,\mathbb C))$
at $\chi_i$ is one-dimensional, we see that the Zariski tangent space of
$X(\mathcal O_i')$ at $\chi_i$ is again of dimension at most
$3$.

As a consequence, $\chi_i$ is a smooth point and the lemma follows.
\end{proof}

We consider the restriction maps
$$
r_i\!:X(\mathcal O'_i)\to X(C')
$$
for $i=1,2$.

\begin{Lemma}\label{Lemma:dim2}
For a sufficiently small neighbourhood $U_i$ of $\chi_i$ in $X(\mathcal
O'_i)$, $r_i(U_i)\cap \mathcal{Y}$ is a $\mathbb C$-analytic 
surface. Moreover the Zariski closure of $\operatorname{Im}(r_i)\cap 
\mathcal{Y}$ is two-dimensional.
\end{Lemma}

\begin{proof}

This follows from the description of the local coordinates around $\chi_i$
given in Lemma~\ref{Lemma:dim3}.

%We claim that the dimension of $X(\mathcal O_i',SL(2,\mathbb C))$ is 
%precisely three and that the trace of the two meridians are independent 
%functions in a neighbourhood of the representation induced by the holonomy 
%(i.e. they define a local submersion onto an open set of $\mathbb C^2$). The 
%lemma follows then from this claim, by imposing that the traces of the two 
%meridians are equal. 

\end{proof}

% % Next we compute the tangent space for $\mathcal O_i'$. 
% % Mayer-Vietoris applied to $\mathcal O_i'$ 
% % and to a neighborhood of the arcs of the tangle gives:
% % $$
% % 0\to H^1(\mathcal O_i;\mathfrak{sl}(2,\mathbb C))\to 
% % H^1(\mathcal O'_i;\mathfrak{sl}(2,\mathbb C))\to 
% % H^1(A_1\cup A_2;\mathfrak{sl}(2,\mathbb C)) \to 0
% % $$
% % Here $A_1$ and the $A_2$ are the annuli that bound the tubular neighbourhood of 
% % the arcs of the tangle. In this piece of Mayer-Vietoris' sequence we omitted 
% % some vanishing groups. For instance 
% % $H^0(\mathcal O_i;\mathfrak{sl}(2,\mathbb C))=0$ because the representation is 
% % irreducible; this gives the leftmost zero. On the other hand, the annuli hane tho homotopy tipe of the circle, thus
% % their second cohomology groups vanish, this explains the rightmost term in the sequence.
% % Finally, the first cohomology group of the tubular neighborhoods of the arcs do not appear, as they
% % are rigid in the orbifold (their fundamnetal group has order two).
% % In this sequence, we shall use that
% % $ H^1(A_1;\cup A_2;\mathfrak{sl}(2,\mathbb C))\cong  H^1(A_1;\mathfrak{sl}(2,\mathbb C))
% % \oplus  H^1(A_2;\mathfrak{sl}(2,\mathbb C))$
% % Since $A_i$ has the homotopy of a circle, $ H^1(A_i;\mathfrak{sl}(2,\mathbb C))\cong\mathbb C$.
% % Hence the dimension of the Zariski tangent space is three. Moreover,
% % since the core circles of $A_1$ and $A_2$ are precisely th emeridians, the trace of the meridians
% % are independent functions.

For hyperbolic orbifolds, Lemma~\ref{Lemma:dim3} can also be understood 
geometrically, by applying the deformation theory of hyperbolic manifolds due 
to Weiss and Hodgson-Kerckhoff \cite{HodgsonKerckhoff,Weiss}.

\begin{Lemma}
\label{Lemma:transversecurves} The analytic germs 
$r_1(U_1)\cap\{t=0\}$ and $r_2(U_2)\cap\{t=0\}$ are curves. In addition 
$\chi_0$ is an isolated point of their intersection.
\end{Lemma}

\begin{proof}
To prove that they are curves, notice that imposing $t=0$ on a representation 
of $\mathcal O'_i$ implies that it factors through a representation of the 
orbifold $\mathcal O_i$. Thus $r_i(U_i)\cap\{t=0\}$ is a curve, by the 
discussion on the variety of representations of $\mathcal O_i$ in the proof of
Lemma~\ref{Lemma:dim3} (for its restriction to $\pi_1(\partial \mathcal O_i)$ 
is non-trivial).

To prove why $\chi_0$ is an isolated point, we need to understand how the 
deformations of representations of $\mathcal O_i$ are seen in $\partial 
\mathcal O_i$, i.e. the orbifold with underlying space $C$, branching locus 
$C\cap L$ and branching indices $2$. We shall consider the induced 
representations on the index $2$ subgroup of $\pi_1(\partial \mathcal O_i)$
consisting of all elements of infinite order. It is a characteristic subgroup 
corresponding to the cyclic branched covering of the torus onto the pillowcase. 
In particular, it makes sense to talk about slopes in $\partial \mathcal O_i$.

When $\mathcal O_i$ is Seifert fibred, its deformations are realised by 
perturbing the base, while keeping the fibre trivial. Thus the slope of 
$\partial \mathcal O_i$ corresponding to the fibre remains constant for every 
perturbation. This proves the lemma when both $\mathcal O_1$ and $\mathcal O_2$ 
are Seifert fibred, because their fibrations do not match, and two different 
slopes on the torus generate its fundamental group.

When $\mathcal O_i$ is hyperbolic, the tangent space to the curve of 
deformations depends on the cusp shape, as observed in the proof of 
Thurston's hyperbolic Dehn filling. If  
$\langle m, l\mid [m,l]\rangle<\pi_1(\partial \mathcal O_i)$ is a presentation 
of the characteristic torsion-free subgroup of index two, then the cusp shape 
is the complex number $\tau$ such that 
$$\rho(l)= \pm
\begin{pmatrix}
    1 & \tau \\
    0 & 1
   \end{pmatrix}
$$
where $\rho$ is the holonomy representation for which 
$$\rho(m)= \pm
\begin{pmatrix}
    1 & 1 \\
    0 & 1
   \end{pmatrix}.
$$
In this case, the deformation satisfies \cite[Lemma~3.11]{Stefano}
$$
\frac{d\operatorname{trace}(l)} {d\operatorname{trace}(m)}=\tau^2.
$$
In particular this proves the lemma for the case of a Seifert fibred orbifold 
glued to a hyperbolic one, because no cusp shape $\tau$ can be zero.

Finally, assume that both $\mathcal O_i$'s are hyperbolic. When replacing $l$ 
by $l\, m$, the previous formula becomes:
$$
\frac{d\operatorname{trace}({l\, m})} {d\operatorname{trace}(m)}=(\tau+1)^2.
$$
So the cusp shape $\tau$ determines the tangent direction of the curve of
deformations, and different values of $\tau$ correspond to different directions of deformation. 
Hence, it suffices to show that the cusp shapes are different, 
for in that case the curves are transversal. In fact, the cusp shapes must be 
different because the induced orientations on the boundary are opposite, hence 
one of the cusp shapes has positive imaginary part, while the other one has 
negative imaginary part.
\end{proof}

\begin{Remark}\label{rem:no match}
The proof of the previous lemma shows that the holonomy of $\mathcal O_1$
restricted to its boundary is not conjugate to the restriction to the boundary
of the holonomy of $\mathcal O_2$. In particular, the induced representations 
on $C'$ are not conjugate even if they have the same character.
\end{Remark}

\begin{Proposition}
\label{proposition:intersection}
The intersection $r_1(U_1)\cap r_2(U_2)\cap\mathcal Y$ is a germ of an analytic 
curve on which $t$ is not constant.
\end{Proposition}

\begin{proof}
Since the germ of $\mathcal Y$ at $\chi_0$ is irreducible by 
Lemma~\ref{Lemma:locallyirreducible}, this is just a dimension count, using 
Lemmata~\ref{Lemma:dim2} and~\ref{Lemma:transversecurves}. In addition, $t$ is nonconstant, again by 
Lemma~\ref{Lemma:transversecurves}.
\end{proof}

Let $\mathcal C$ be the algebraic curve which is the Zariski closure of the 
intersection $r_1(U_1)\cap r_2(U_2)\cap\mathcal Y$. Note that it contains 
$\chi_0$.

\begin{Lemma}
\label{lemma:lift}
There exists an algebraic curve 
$\mathcal D\subset X(S^3\setminus L)$ such that the 
restriction from $\pi_1(S^3\setminus L)$ to $\pi_1(C')$ induces a non-trivial 
regular map $\mathcal D\longrightarrow \mathcal C$.
\end{Lemma}

\begin{proof}
Let $\mathcal C^{irr}$ denote the set of irreducible characters of 
$\mathcal C$, which is Zariski open. By Lemma~\ref{lemma:irrdeformations}, 
$\mathcal C^{irr}\neq\emptyset$. For $i=1,2$, let 
$\mathcal C_i\subset r_i^{-1}(\mathcal C)\subset X(\mathcal O'_i)$ 
be an irreducible curve which contains the holonomy character $\chi_i$ of 
$\mathcal O_i$. If such curve is chosen generically then we can assume that 
$r_i\! : \mathcal C_i\to  \mathcal C$ in non constant. Let 
$p_i\!: X(S^3\setminus L)\to X(M_i')$ be the morphism induced by the inclusion 
of fundamental groups. Notice that $X(M'_i)\supset X(\mathcal O'_i)$.
We consider the algebraic set:
$$
\mathcal D=\{ \chi\in p_1^{-1}(\mathcal C_1)\cap p_2^{-1}(\mathcal C_2)
\mid r_1(p_1(\chi))= r_2(p_2(\chi))\}.
$$
The intersection 
$Z=r_1(\mathcal C_1)\cap r_2(\mathcal C_2) \cap \mathcal C^{irr}$ is 
a nonempty (see Lemma~\ref{lemma:irrdeformations}) Zariski open 
subset of $\mathcal C$. In addition, every point of $Z$ is the restriction of a 
\emph{unique} character in $\mathcal D$, by Lemma~\ref{lemma:glueirreducible}. 
So if we choose an irreducible component of $\mathcal D$ containing a point 
with image in $Z$, then the restriction to this component of 
$r_1\circ p_1= r_2\circ p_2$ is non-trivial and the component is a curve. 
\end{proof}

\begin{proof}[Proof of Theorem~\ref{theorem:onesphere}]

With the notation of the previous lemma, choose a sequence of points in
$r_1(\mathcal C_1\cap U_1)\cap r_2(\mathcal C_2\cap U_2)\subset\mathcal C$ 
which converges to $\chi_0$. The sequence lifts to sequences in
$\mathcal C_i$ which converge to $\chi_i$, $i=1,2$. Up to passing to a
subsequence, all these sequences are induced by a sequence in $\mathcal D$. We 
claim that such sequence has no accumulation point in $\mathcal D$. This 
follows from the fact that the representations inducing the characters $\chi_i$ 
do not match (Remark~\ref{rem:no match}). On the other hand, the restrictions 
of these characters to the $M_i$'s are induced by representations that converge 
to the holonomy representations of $\mathcal O_i$. Now apply 
Lemma~\ref{lemma:suffideal}.
\end{proof}

\section{Bonahon-Siebenmann families}
\label{section:bs families}

Let $L$ be a hyperbolic link in the $3$-sphere and consider the orbifold whose
underling topological space is $\mathbf S^3$ and whose singular set is $L$ with
branching indices of order $2$ or $3$. Assume that its Bonahon-Siebenmann
decomposition is non trivial. Note that if we replace the ramification indices
of order $3$ by higher orders of ramification, the decomposition will still be 
the same, and even the type of geometric structure will not change. The Conway 
spheres of the decomposition cut the orbifold into pieces which are either 
hyperbolic or Seifert fibred. Note that we can increase the order of 
singularity of the components of $L$ which do not meet the toric family without 
changing the decomposition (although the geometries involved may change) and so 
that no Seifert piece has a Euclidean base. This follows from 
Lemma~\ref{lemma:euclidean} which holds for decompositions with an arbitrary 
number of pieces. From now on we shall thus assume, without loss of generality, 
that the order of singularity of the components which do not meet the 
Bonahon-Siebenmann family is $3$, while the order of singularity of the
remaining ones is $2$.

Assume a Bonahon-Siebenmann family as above is given. Let $C_1,\dots,C_k$ be 
the Conway spheres of the family. We denote by $C_l'$, $l=1,\dots k$, the 
intersection of $C_l$ with the exterior of the link. Let $M_1,\dots,M_{k+1}$ be 
the $k+1$ connected components of $S^3\setminus(C_1\cup\cdots\cup C_k)$ and 
$M'_1,\dots,M'_{k+1}$ their respective intersections with the exterior
of the link. Finally let $\mathcal O_1,\dots,\mathcal O_{k+1}$ denote the 
orbifolds with the chosen orders of ramification, corresponding to 
$M_1,\dots,M_{k+1}$, and $\mathcal O'_1,\dots,\mathcal O'_{k+1}$ the orbifolds 
obtained by removing the branching components of order $2$, i.e. those that 
meet $C_1\cup\cdots\cup C_k$. 

In the same spirit of what was done in the single Conway sphere case, we denote
by $\mathcal Y_l$, $l=1,\dots, k$,  the subvariety of $X(\mathcal C_l')$ 
obtained by imposing that all meridians have the same trace. It has the same 
equation as in (\ref{eq:hypersurface}). 

We now need to prove analogues of Lemmata~\ref{Lemma:dim3}, \ref{Lemma:dim2} 
and \ref{lemma:lift} in this new setting.

For $i=1,\dots,k+1$, let $b_i$ the number of boundary components of $M_i$. Note
that the number of arcs with ramification of order $2$ in $\mathcal O_i$ is
$2b_i$.

Lemmata~\ref{Lemma:dim3} and \ref{Lemma:dim2} will be replaced by the following 
remark in the hyperbolic case and by the subsequent lemma which deals with the 
Seifert fibred case.

\begin{Remark}\label{rem:hyperbolique}
Assume $\mathcal O_i$ is hyperbolic and let $\chi_i$ be a character in 
$X(\mathcal O'_i)$ that is a lift of the holonomy representation for 
$\mathcal O_i$. The very same proof of Lemma~\ref{Lemma:dim3} shows that 
$\chi_i$ is a smooth point and the traces of the $2b_i$ meridians and $b_i$ 
peripheral elements of $\mathcal O_i$, one for each boundary component, 
constitute a local coordinate system in a neighbourhood of $\chi_i$. 
Lemma~\ref{Lemma:dim2} (with dimension $b_i+1$ instead of $2$) follows at once.
\end{Remark}

Assume now that $\mathcal O_i$ is Seifert fibred. The hyperbolic structure of
the base may be not unique if the base is large for, in this case, the 
Teichm\"uller space has positive dimension. In particular, the choice for the
holonomy character is not unique. However, it is easy to prove that there is a
preferred structure. The base is a polygonal orbifold with mirror boundary, 
some of whose corners are cusps, and with at most one cone point in the 
interior. The preferred structure corresponds to the situation in which the 
base polygonal orbifold admits an inscribed circle tangent to every side. The
character $\chi_i\in X(\mathcal O'_i)$ will denote a lift of the holonomy 
of this preferred structure for $\mathcal O_i$. The interest of this specific
$\chi_i$ comes from the fact that it admits deformations which correspond to
hyperbolic cone structure as proved in \cite{Porti}. This means that $\chi_i$ 
is contained in an irreducible component $W_i$ of the variety of $X(\mathcal
O'_i)$ which contains points representing hyperbolic cone structures. 
Lemma~\ref{lemma:seifert} investigates some properties of the subvariety
$V_i$ of $W_i$, obtained by imposing that all meridians have the same trace.

Let $C_{l_1},\dots, C_{l_{b_i}}$ be the boundary components of $M_i$. 
We shall follow the usual convention that the character variety of a disjoint
union is the product of the character varieties of the components. Let 
$$
r_{\partial_i}\! : X(\mathcal O'_i)\longrightarrow
X(C'_{l_1}\cup\cdots\cup C'_{l_{b_i}})=
X(C'_{l_1})\times\cdots\times X( C'_{l_{b_i}})
$$ 
be the projection induced by restriction. 
 
\begin{Lemma}\label{lemma:seifert}
For a sufficiently small neighbourhood $U_i$ of $\chi_i$ in $V_i$, 
$r_{\partial_i}(U_i)$ is a $\mathbb C$-analytic $(b_i+1)$-variety and its
Zariski closure is again $(b_i+1)$-dimensional. 
\end{Lemma}

\begin{Remark}
When the base of the fibration of $\mathcal O_i$ is large, namely when
the Teichm\"uller space of this base is non-trivial,
the character $\chi_i$ is contained in another component of $X(\mathcal O_i')$.
This other component can be identified to the character variety of $\mathcal 
O_i$, whose points correspond to deformations of the base of the fibration, and 
it is locally the complexification of its Teichm\"uller space. In this other 
component $t$ is constant, i.e. $t=0$.
\end{Remark}

\begin{proof}[Proof of Lemma~\ref{lemma:seifert}]
Consider the map $\pi \!:\! U_i\subset V_i \to \mathbb C^{b_i+1}$, whose 
components are $t$ and the trace of peripheral elements corresponding to the 
boundary components of the base of the fibration, one for each boundary 
component of $M_i$. By the analogue of \cite[Claim~6.6]{Porti} (and with 
exactly the same proof), $\pi^{-1}(\pi(\chi_i)))=\{\chi_i\}$. Since $\pi$ 
factors through $r_{\partial_i}$, by Remmert's proper map theorem 
$r_{\partial_i}(U_i)$ is a $\mathbb C$-analytic subvariety \cite{Remmert} 
(cf.~\cite[Thm.~V.4.A]{Whitney} or \cite[Thm.~V.C.5]{GunningRossi}). In 
addition, by the openness principle, $\pi(U_i)$ is a neigbourhood of 
$\pi(\chi_i)$ in $\mathbb C^{b_i+1}$. This proves that $r_{\partial_i}(U_i)$ 
has dimension $b_i+1$. 
\end{proof}

Given $\chi_i\in X(\mathcal O'_i)$ a lift of the holonomy as above, we take 
$V_i$ to be the irreducible subvariety defined as follows.

\begin{itemize}

\item 
When $\mathcal O_i$ is hyperbolic, $V_i$ is defined to be the unique 
irreducible component of the subvariety of $X(\mathcal O_i')$ obtained by 
imposing that the traces of all meridians are the same, and containing 
$\chi_i$. This is well defined by Remark~\ref{rem:hyperbolique}.

\item  
When $\mathcal O_i$ is Seifert fibred, then $V_i$ is as in 
Lemma~\ref{lemma:seifert}.

\end{itemize}
By Remark~\ref{rem:hyperbolique} and Lemma~\ref{lemma:seifert}, for every 
$i=1,\ldots,k+1$ there exists an open neighborhood $\chi_i\in U_i\subset V_i$
such that $r_{\partial_i}(U_i)$ is a $\mathbb C$-analytic $(b_i+1)$-variety.

Now group the components $M_1,\ldots, M_{k+1}$ in two families, so that if 
$M_i$ and $M_j$ are in the same family and $i\neq j$, then 
$M_i\cap M_j=\emptyset$. This is possible because the dual graph to the 
decomposition along Conway spheres is a tree. Up to permuting the indices, we 
may assume that the first family is $M_1,\ldots, M_{k_0}$ and the second one 
$M_{k_0+1},\ldots,M_{k+1}$.
We denote:
$$
M_+=M_1\cup\cdots\cup M_{k_0}\quad\textrm{ and }\quad 
M_-=M_{k_0+1}\cup\cdots\cup M_{k+1}.
$$
Similarly one defines $\mathcal O_{\pm}$ and $\mathcal O'_{\pm}$. Notice that, 
since $\mathcal O'_{\pm}$ is not connected,
$$
X(\mathcal O_+')=X(\mathcal O_1')\times \cdots\times X(\mathcal O_{k_0}'),
$$
and analogously for $X(\mathcal O_-')$. Let
$$
\chi_+=(\chi_1,\ldots,\chi_{k_0}) \in X( \mathcal O_+') \quad
\textrm{ and }\quad 
\chi_-=(\chi_{k_0+1},\ldots,\chi_{k+1})\in X( \mathcal O_-') .
$$
Remember that $M_i$ has $b_i$ boundary components, then
$$
b_1+\cdots+ b_{k_0}=b_{k_0+1}+\cdots+ b_{k+1}=k
$$
Consider the product
$$
\mathcal Y=\mathcal Y_1\times\cdots \mathcal Y_k,
$$
and the restriction maps whose components are the $ r_{\partial_{i}}$
$$
r_{\pm}\!:\!          X(\mathcal O_{\pm}')\to \mathcal Y.
$$

Take $V_+=V_1\times\cdots\times V_{k_0}$ and 
$V_-=V_{k_0+1}\times\cdots\times V_{k+1}$. By Remark~\ref{rem:hyperbolique} and 
Lemma~\ref{lemma:seifert}, there exist $U_{\pm}\subset V_\pm$ neighbourhoods of 
$\chi_{\pm}$ such that $r_{\pm}(U_{\pm})$ is an analytic variety of dimension 
$b_{\pm}$, where
$$
b_+ = \sum_{i=1}^{k_0} (b_i+1)= k+k_0\quad\textrm{ and }\quad 
b_-=\sum_{i=k_0+1}^{k+1} (b_i+1)= 2 k+1 -k_0.
$$

Notice that $b_++b_-=3 k+1= \dim \mathcal Y+1$. Now the analogue of 
Lemma~\ref{Lemma:transversecurves}, that uses the same arguments and the 
results of Remark~\ref{rem:hyperbolique} and Lemma~\ref{lemma:seifert}, is the 
following lemma. 

\begin{Lemma}
\label{Lemma:transversevarieties}
The analytic varieties $r_{+}(U_{+})\cap \{t=0\}$ 
and $r_{-}(U_{-})\cap \{t=0\}$ 
have both dimension $k$. 
In addition, $\chi_0 $ is an isolated point of the intersection 
$r_{+}(U_{+})\cap r_{-}(U_{-})\cap \{t=0\}$.
\end{Lemma}

Notice that the traces of meridians of $\mathcal Y_l$ and $\mathcal Y_{l'}$
are different variables, for $l\neq l'$. In $\mathcal Y\cap \{t=0\}$ we require 
that they are all the same and equal to zero. In particular
$\dim (\mathcal Y\cap \{t=0\})=2 k$. On the other hand, by construction, in 
$r_{+}(U_{+})\cap r_{-}(U_{-})$ the traces of the meridians of $\mathcal Y_l$ 
are the same.

As a corollary of Lemma~\ref{Lemma:transversevarieties}, we obtain the analogue 
of Proposition~\ref{proposition:intersection}, using 
Lemma~\ref{Lemma:locallyirreducible}.

\begin{Proposition}
\label{proposition:HDintersection}
The intersection 
$r_+(U_+)\cap r_-(U_-)\cap\mathcal Y$ is a curve on which $t$ is not constant.
\end{Proposition}

Exactly the same proof as that of Lemma~\ref{lemma:lift} gives:

\begin{Lemma}\label{lemma:multilift}
There exists an algebraic curve $\mathcal D\subset X(S^3\setminus L)$ such that 
the restriction from $\pi_1(S^3\setminus L)$ to 
$\pi_1(C_1'),\ldots,\pi_1(C_k')$ induces a non-trivial regular map 
$\mathcal D\longrightarrow \mathcal C$.
\end{Lemma}

Finally, the proof of Theorem~\ref{theorem:morespheres} is the same as the 
proof of Theorem~\ref{theorem:onesphere}.

\begin{footnotesize}

%}

\bibliographystyle{plain}

%\begin{thebibliography}{XXXXX}
%\bibliography{conway}
%\end{thebibliography} 

\textsc{Universit\'e de Provence, LATP UMR 6632 du CNRS}
	
\textsc{CMI, Technop\^ole de Ch\^ateau Gombert}

\textsc{39, rue F. Joliot Curie, 13453 Marseille CEDEX 13, France}
	
{paoluzzi@cmi.univ-mrs.fr}

\medskip

\textsc{Departament de Matem\`atiques, Universitat Aut\`onoma de Barcelona.}

\textsc{08193 Bellaterra, Spain}

{porti@mat.uab.es}

\end{footnotesize}

\end{document}